\documentclass[11pt]{article}
\usepackage{amssymb,amsmath,amsthm}
\usepackage{graphicx}
\usepackage{epsfig}
\usepackage{color}
\usepackage[normalem]{ulem}

\newtheorem{corollary}{Corollary}[section]
\newtheorem{lemma}[corollary]{Lemma}
\newtheorem{proposition}[corollary]{Proposition}

\newcommand{\Prob} {{\mathbb P}}
 
\newcommand{\E}{{\mathbb E}}

\newcommand{\dist}{{\rm dist}}
\newcommand{\x}{{\bf x}}
\newcommand{\y}{{\bf y}}

\newcommand {{\ball}} {{\mathcal B}}

\def \p {\partial}

\def \Half {{\mathbb H}}

\def \eset {\emptyset}

\def \paths {{\cal K}}

\newcommand{{\slegreen}} {{\mathbb G}}

\def \linehere { {\hrule}}
\def \labove { \mtwo \linehere \linehere \linehere \ms   }

\def \lbelow {{\ms \linehere \linehere \linehere \mtwo}}

\def \mtwo {{\medskip \medskip}}
\def \ms {{\medskip}}
\newcommand {{\whoknows}} {{\mathcal C}}

\newenvironment{definition}[1][Definition]{\begin{trivlist}
\item[\hskip \labelsep {\bfseries #1}]}{\end{trivlist}}

\newenvironment{advanced}
{ \labove \begin{quote} \begin{small}}
{ \end{small}\end{quote} \lbelow }

\def \begad{\begin{advanced}}
\def \endad{\end{advanced}}

\newcommand {\exc} {{\mathcal E}}

\newcommand{{\rect}} {{\mathcal R}}

\newcommand {{\hulls}}  {{\mathcal M}}
\newcommand {{\stuff}} {{\rm fill }}

\newcommand{{\dyadic}}  {{\mathcal Q}}

\newcommand{{\partition}} {{\mathcal P}}

\newcommand{{\azero}}  {\sim_0}
\newcommand{{\api}} {\sim_\pi} 
\newcommand {{\phiset}}  {{\mathcal X}}
 \newcommand {{\crad}} {{\rm crad}}
\newcommand {{\eb}}  {{\bf e}}
\newcommand {{\cent}} {{\bf c}}
\newcommand  \z {{\bf z}}

\newcommand {{\curves}} {\paths}
 
\newcommand {{\measures}} {{\mathcal M}}
\newcommand{{\brown}}  {{\nu}}
\newcommand{{\osc}}  {{\rm osc}}
\newcommand  {{\rbrown}} {{\mu}}
\newcommand{{\lcurves}}{{\curves_L}}
\newcommand{{\bcurves} } {\curves^{\rm bub}}
\newcommand {{\bbrown}}{{\nu^{\rm bub}}}
\newcommand {{\n}} {{\bf n}}

\newcommand{{\brownroot}} {\brown^{\rm r}}
\newcommand{{\dhalf}}{{\mathcal H}}

\newcommand {{\bgamma}}  {{\boldsymbol \gamma}}

\def \sle {SLE_\kappa}

\title{On the smoothness of the partition function for
multiple Schramm-Loewner evolutions}
\author{Mohammad Jahangoshahi\\
Gregory F. Lawler\thanks{Research supported by
National Science Foundation grant DMS-1513036} }

\begin{document}

\maketitle
 
\begin{abstract}
We consider the measure on multiple chordal Schramm-Loewner evolution ($SLE_\kappa$)
curves.  We 
establish a derivative estimate and use it to
give a direct proof that the partition function
is $C^2$ if $\kappa < 4$.  

\end{abstract}

\section{Introduction}


The (chordal) Schramm-Loewner evolution with
parameter $\kappa >0$   ($SLE_\kappa$) is a measure on
curves connecting two distinct boundary points $z,w$ of a simply
connected domain $D$.  As originally defined by Schramm \cite{Oded}, this
is a probability measure on paths. If $\kappa \leq 4$,   the measure  is supported on simple
curves that do not touch the boundary. 
Although   Schramm \cite{Oded}
originally defined $\sle$ as a probability measure,
if $\kappa \leq 4$ and $z,w$ are locally
analytic boundary points, it is natural to consider $\sle$ as a
finite 
measure $\mu(z,w)$  with partition function, that is,
as a  measure with total mass $\Psi_D(z,w)=H_D(z,w)^b$. 
Here $H$ denotes
the boundary Poisson kernel normalized so that $H_\Half(0,x) = x^{-2}$
and $b = (6-\kappa)/2\kappa$ is the boundary scaling
exponent. If $f: D \rightarrow f(D)$
  is a conformal transformation, then
  \[     H_D(z,w) = |f'(z)| \, |f'(w)| \, H_{f(D)}(f(z),f(w)).\]
 
 There are several reasons for considering $\sle$ as a measure with a total mass. First, $\sle$ is known to be the scaling limit of various two-dimensional discrete models that are considered as measures with partition functions, and hence
  it is natural to consider the (appropriately normalized)
  partition function in the scaling limit. Second, the ``restriction property'' or ``boundary
  perturbation'' can be described more naturally for the 
 nonprobability measures; see \eqref{apr27.1} below.  This description leads to one
 way to define $\sle$ in multiply connected domains or, as is
 important for this paper, for multiple $\sle$ in a simply connected
 domain.   See \cite{Parkcity, Annulus} for more information.
  We write $\mu^\#_D(z,w) = \mu_D(z,w)/  
\Psi_D(z,w)$ for the probability measure which is well
defined even for rough boundaries.  

  The definition of the measure on multiple $\sle$ paths immediately
  gives a partition function defined as the total mass 
  of the measure.  The measure on multiple $\sle$ 
  paths  \[   \bgamma = (\gamma^1,\ldots,
    \gamma^n)\] has been constructed in \cite{Dub,KL,Parkcity}. Even though the definition in \cite{KL} is given for the so-called ``rainbow'' arrangement of the boundary points, it can be easily extended to the other arrangements \cite{Dub, Stat}.
   One can see that unlike $\sle$ measure on single curves, conformal invariance and domain Markov property do not uniquely specify the measure when $n\geq 2$. This definition makes it unique
    by requiring the measure to satisfy the restriction property, which is explained in Section \ref{defs}.
   
Study of the multiple $\sle$ measure involves characterizing the partition function. For $n=2$, the partition function is explicitly given in terms of the hypergeometric function. 
For $n\geq 3$, the goal is to characterize the partition function by a particular second-order PDE. 
However, it does not directly follow from the definition that the partition function is $C^2$. There are two main approaches to address this problem.
 One approach is to show that the PDE system has a solution and use it to describe the partition function. 
In \cite{Dub}, it is shown that a family of integrals taken on a specific set of cycles satisfy the required PDE system. 
In \cite{Eve}, conformal field theory and partial differential equation techniques such as H{\"o}rmander's theorem are used to show that the partition function satisfies the PDE system.
 The other approach, which is the one we take in this work, is to directly prove that the partition function is $C^2$. Then It\^{o}'s formula can be used next to show that the partition function satisfies the PDEs.
 
   The basic idea
of our proof
is to interchange derivatives and expectations in expressions for the
partition function.  This interchange needs justification and we prove
an estimate about $\sle$ to justfiy this.

Here we summarize the paper.  We finish this introduction by
reviewing examples of partition functions for $\sle$.
Definitions and properties of multiple $\sle$ and the outline of the proof are given in Section \ref{defs}. 
Section \ref{japansec} includes an estimate for $\sle$ using techniques similar to the ones in \cite{Japan}. 
Proof of Lemma \ref{mar19.lemma1}, which explains estimates for derivatives of the Poisson kernel is given in Section \ref{lemmasec}.

\subsection{Examples}\label{exmaples}
 
\begin{itemize}
\item \textbf{$\sle$ in a subset of $\Half$.} Let $\kappa\leq 4$ and suppose $D\subset \Half$ is a simply connected domain such that $K=\Half\setminus D$ is bounded and $\dist(0,K)>0$. Also, assume that $\gamma$ is parameterized with half-plane capacity. By the restriction property we have
\begin{equation}  \label{apr27.1}
\frac{d\mu_D(0,\infty)}{d\mu_\Half(0,\infty)}(\gamma)=1\{\gamma\cap K=\emptyset\}\exp\left\{\frac{\textbf{c}}{2}m_\Half(\gamma,K)\right\},
\end{equation}
where $m_\Half(\gamma,D)$ denotes the Brownian loop measure of the loops that intersect both $\gamma$ and $K$ and 
\[
\cent =\frac{(6-\kappa)(3\kappa-8)}{2\kappa}
\]
is the \emph{central charge}.
 We normalize the partition functions, so that $\Psi_\Half(0,\infty)=1$. For an initial segment of the curve $\gamma_t$, let $g_t:\Half\setminus\gamma_t\to\Half$ be the unique conformal transformation with $g_t(z)=z+\text{o}(1)$ as $z\to\infty$. Then 
\[
\partial_tg_t(z)=\frac{a}{g_t(z)-U_t},
\]
where $a=2/\kappa$ and $U_t$ is a standard Brownian motion. Suppose $\gamma_t\cap K=\emptyset$ and let $D_t=g_t(D\setminus\gamma_t)$. One can see that 
\[
m_\Half(\gamma_t,K)=-\frac{a}{6}\int_0^tS\Phi_s(U_s)ds,
\]
where $S$ denotes the Schwarzian derivative and $\Phi_s(U_s)=H_{D_s}(U_s,\infty)$. It follows from conditioning on $\gamma_t$ that 
\[
M_t=\exp\left\{\frac{\textbf{c}}{2}m_\Half(\gamma_t,K)\right\}\Psi_{D_t}(U_t,\infty)
\]
is a martingale. We assume the function $V(t,x)=\Psi_{D_t}(x,\infty)$ is $C^2$ for a moment. Therefore, we can apply   It\^{o}'s formula and we get
\[
-\frac{a\textbf{c}}{12}V(t,U_t)\,S\Phi_t(U_t)+{\partial_t V(t,U_t)}+\frac{1}{2}\partial_{xx} V(t,U_t)=0.
\]
Straightforward calculation shows that $V(t,x)=H_{D_t}(x,\infty)^b$ is $C^2$ and satisfies this PDE. Here, $b$ is the \emph{boundary scaling exponent}
\[
b=\frac{6-\kappa}{2\kappa}.
\]
\item \textbf{Other examples.} Similar ideas were used in \cite{KL} to describe the partition function of two $\sle$ curves with a PDE.
 Differentiability of the partition function was justified using the explicit form of the solution in terms of the hypergeometric function.
  The PDE system in \cite{Annulus} characterizes the partition function of the annulus $\sle$. 
  That PDE is more complicated and one cannot find an explicit form for the solution. 
  In fact, it is not easy to even show that the PDE has a solution. Instead, it was directly proved that the partition function is $C^2$ and It\^{o}'s formula was used to derive the PDE.

\end{itemize}

\section{Definitions and Preliminaries}\label{defs}
We will consider  the multiple $\sle$ measure
  only for $\kappa \leq 4$
 on simply connected domains $D$ and distinct
locally analytic boundary points  $\x = (x_1,\ldots,x_n),
\y = (y_1,\ldots,y_n)$. 
The measure is supported on $n$-tuples of curves
 \[   \bgamma = (\gamma^1,\ldots,
    \gamma^n),\]
 where $\gamma^j$ is a curve connecting $x_j$ to $y_j$ in $D$.
 If $n = 1$, then $\mu_D(x_1,y_1)$ is $SLE_\kappa$ from $x_1$ to $y_1$
 in $D$  with total mass $H_D(x_1,y_1)^b
 $ whose corresponding probability measure
  $\mu^\#_D(x_1,y_1) = \mu_D(x_1,y_1)/H_D(x_1,y_1)^b $ is
  (a time change of)
  $SLE_\kappa$ from $x_1$ to $y_1$ as defined by  
  Schramm.

 \begin{definition} If $\kappa \leq 4$ and
  $n \geq 1$,
 then $\mu_D(\x,\y)$ is the measure absolutely continuous with respect
 to $\mu_D(x_1,y_1) \times \cdots \times \mu_D(x_n,y_n)$ with
 Radon-Nikodym derivative
 \[         Y(\bgamma) :=    I(\bgamma) \, \exp\left\{
 \frac \cent 2\sum_{j=2}^n m\left[K_j(\bgamma) \right]   \right\}. \]
 Here $\cent = (6-\kappa)(3\kappa - 8)/2\kappa $ is the
 central charge, $I(\bgamma)$ is the indicator function of the event
 \[     \{\gamma^j \cap \gamma^k = \eset , 1 \leq j < k \leq n\}, \]
 and $m\left[K_j(\bgamma)\right]$ denotes the Brownian loop measure of loops
 that intersect at least $j$ of the paths $\gamma^1,\ldots,\gamma^n$.
 \end{definition}
 
  Brownian loop measure is a measure on (continuous) curves $\eta:[0,t_\gamma]\to\mathbb{C}$ with $\eta(0)=\eta(t_\eta)$. 
  Let $\nu^\#(0,0;1)$ be the law of the Brownian bridge starting from 0 and returning to 0 at time 1. 
  Brownian loop measure can be considered as the measure
\[
m=\text{area}\times\left(\frac{1}{2\pi t^2}dt\right)\times\nu^\#(0,0;1)
\]
on the triplets $(z,t_\eta,\tilde{\eta})$, where $\tilde\eta(t)=t_\eta^{1/2}\eta(t/t_\eta)$ for $t\in[0,1]$. 
For a domain $D\subset\mathbb{C}$, we denote the restriction of $m$ to the loops $\eta\subset D$ by $m_D$.  
One important property of $m_D$ is conformal invariance. More precisely, if $f:D\to f(D)$ is a conformal transformation, then 
\[
f\circ m_D= m_{f(D)},
\]
where $f\circ m_D$ is the pushforward measure.

 Note that if $\sigma$ is a permutation of $\{1,\ldots,n\}$ and
 $\bgamma_\sigma = (\gamma^{\sigma(1)},\ldots,\gamma^{\sigma(n)})$, 
 then $Y(\bgamma) = Y(\bgamma_\sigma)$. The partition function
 is the total mass of this measure
 \[  \Psi_D(\x,\y) = \| \mu_D(\x,\y)\|. \]
 We also write
 \[  \tilde \Psi_D(\x,\y) = \frac{\Psi_D(\x,\y)}{\prod_{j=1}^n H_D(x_j,y_j)^b},\]
 which can also be written as
 \[   \tilde \Psi_D(\x,\y)  = \E[Y] , \]
 where the expectation is with respect to the probability measure
 $\mu_D^\#(x_1,y_1) \times \cdots \times \mu_D^\#(x_n,y_n)$.
 Note that  $\tilde \Psi_D(\x,\y)$ is a conformal invariant,
 \[   f\circ \tilde \Psi_D(\x,\y) = \tilde \Psi_{f(D)}
  (f(\x),f(\y)), \]
  and hence is well defined even if the boundaries are rough.
   Since $SLE_\kappa$ is
 reversible \cite{Zhan}, interchanging $x_j$ and $y_j$ does not change
 the value. 
 
 To compute the partition function we use an alternative
 description of the measure  $\mu_D(\x,\y)$.  We will give
 a recursive definition.
 \begin{itemize}
 \item  For $n=1$, $\mu_D(x_1,y_1)$ is the usual $SLE_\kappa$
 measure with total mass $H_D(x_1,y_1)^b$.
 \item  Suppose the measure has been defined for all $n$-tuples
 of paths.  Suppose $\x = (x',x_{n+1}), \y = (y',y_{n+1})$
 are given and write an $(n+1)$-tuple of paths as
 $\bgamma = (\bgamma',\gamma^{(n+1)})$.  
 \begin{itemize}
 \item The marginal
 measure on $\bgamma'$ induced by
 $\mu_D(\x,\y)$  is absolutely continuous with
 respect to $\mu_D(\x',\y')$ with Radon-Nikodym derivative
 $H_{\tilde D}(x_{n+1},y_{n+1})^b$.  Here $\tilde D$ is the
 component of $D \setminus  \bgamma'$ containing
 $x_{n+1},y_{n+1}$ on its boundary.  (If there is no such
 component, then we set $H_{\tilde D}(x_{n+1},y_{n+1}) = 0$
 and $\mu_D(\x,\y)$ is the zero measure.)
 \item Given $\bgamma'$, the curve $\gamma^{n+1}$
 is chosen using the probability distribution
 $\mu^\#_{\tilde D}(z_{n+1},y_{n+1})$.
 
 \end{itemize}
 
 \end{itemize}

 One could try to use this description of the measure as the
 definition, but it is not obvious that it is consistent.  However, one can see that the first definition satisfies this property using the following lemma.
\begin{lemma}
Let $\bgamma$ denote a $(n+1)$-tuple of paths which we write as
 $\bgamma = (\bgamma',\gamma^{(n+1)})$, and let $\tilde{D}$ be the connected component of $D\setminus\bgamma'$ containing the end points of $\gamma^{(n+1)}$ on its boundary. Then 
 \[
 \sum_{j=2}^{n+1} m\left[K_j(\bgamma)\right]=\sum_{j=2}^n m\left[K_j(\bgamma')\right]+m_D(\gamma^{(n+1)},\,D\setminus\tilde{D}).
 \]
\end{lemma} 
\begin{proof}
Let $K^1_j(\bgamma)$ denote the set of loops in $K_j(\bgamma)$ that  intersect $\gamma^{(n+1)}$ and let $K^2_j(\bgamma)$ denote the set of loops that do not intersect  $\gamma^{(n+1)}$. Then 
\begin{equation}\label{eq_brwn1}
m\left[K^1_2(\bgamma)\right]=m_D(\gamma^{(n+1)},\,D\setminus\tilde{D}).
\end{equation}
Note that $K^1_j(\bgamma)$ is equivalent to the set of loops in $D$ that intersect  $\gamma^{(n+1)}$ and at least $j-1$ paths of $\bgamma'$. Moreover, $K^2_j(\bgamma)$ is equivalent to the set of loops that intersect at least $j$ paths of $\bgamma'$, but do not intersect $\gamma^{(n+1)}$. Therefore,
\[
K_j(\bgamma')=K^1_{j+1}(\bgamma)\cup K^2_{j}(\bgamma).
\]
Now the result follows from this, the fact that $K^2_{n+1}(\bgamma)=\emptyset$ and \eqref{eq_brwn1}.
\end{proof}
 We can also take the marginals in a different order. For
 example, we could have defined the recursive step above as follows.
 \begin{itemize}
 \item  The marginal measure on $\gamma^{n+1}$ induced
 by $\mu_D(\x,\y)$ is absolutely continuous with respect
 to $\mu_D(x_{n+1},y_{n+1})$ with Radon-Nikodym
 derivative $\Psi_{\tilde D}(\x',\y')$ where
 $\tilde D = D \setminus \gamma$.  (It is possible that
 $\tilde D$ has two separate components in which case we
  multiply the partition functions on the two components.)
 \end{itemize}

We will consider boundary points on the real line. 
We write just $H,\Psi,\tilde \Psi,\mu,\mu^\#$ for $H_\Half,
\Psi_\Half, $ $\tilde \Psi_\Half, $ $\mu_\Half,$ $\mu^\#_\Half$;
and note that
 \[ \tilde  \Psi(\x,\y) =
 \E\left[   Y\right] = \Psi(\x,\y) \,  \prod_{j=1}
   ^n |y_j - x_j|^{ 2b}, \]
   where the expectation is with respect to the
   probability measure  
\[  \mu^\#(x_1,y_1) \times \cdots \times  \mu^\#(x_n,y_n)
.\]  

\begin{itemize}

\item If $n = 1$, then $Y \equiv 1$ and $\tilde \Psi(\x,\y) = 1$.

\item  For $n = 2$ and $\bgamma = (\gamma^1,\gamma^2)$,
then
\[    \E[Y \mid \gamma^1] = \left[\frac{H_{D\setminus \gamma^1}
  (x_2,y_2)}{H_D(x_2,y_2)} \right]^b.\]
 The right-hand side is well defined even for non smooth boundaries
 provided that $\gamma^1$ stays a positive distance from $x_2,y_2$.
 In particular,
 \[   \E[Y] = \E\left[\E(Y\mid \gamma^1)\right]
     = \E\left[\left(\frac{H_{D\setminus \gamma^1}
  (x_2,y_2)}{H_D(x_2,y_2)} \right)^b\right]  \leq 1.\]
  If $8/3 < \kappa \leq 4$, then $\cent >0$ and $Y > 1$
  on the event $I(\bgamma)$ so the inequality 
  $\E[Y] \leq 1$ is not obvious.
  
  \item  More generally,  if $\bgamma = (\bgamma',\gamma^{n+1})$,
  \[    \E[Y \mid \bgamma']
       = Y(\bgamma') \,  \left[\frac{H_{D\setminus \bgamma'}
  (x_{n+1},y_{n+1})}{H_D(x_{n+1},y_{n+1})} \right]^b   \leq Y(\bgamma').\]
 Using this we see that $\tilde \Psi_D(\x,\y) \leq 1$.
 
\item For $n = 2$, if $x_1 = 0,
y_1 = \infty, y_2 = 1$ and $x_2 = x$ with $0 < x  <1$, we have (see,
for example, \cite[(3.7)]{KL})
\begin{equation}  \label{mar18.4}
 \tilde \Psi(\x,\y) = \phi(x) := \frac{\Gamma(2a) \, \Gamma(6a-1)}
 {\Gamma(4a) \, \Gamma(4a-1)} \, x^a\,
    F(2a,1-2a,4a;x) , 
    \end{equation}
    where $F =$ $_2F_1$ denotes the hypergeometric function
    and $a = 2/\kappa $.  This is computed by
  finding
  \[   \E\left[  H_{\Half\setminus \gamma^1}
   (x,1)^b\right].\]
 In fact, this calculation is valid for $\kappa < 8$ if it
 is interpreted as
 \[   \E\left[  H_{\Half\setminus \gamma^1}
   (x,1)^b; H_{\Half\setminus \gamma^1}
   (x,1) > 0\right].\]
    
  \end{itemize}
It will be useful to write
 the conformal invariant \eqref{mar18.4}  
 in a different way.  If $V_1,V_2$ are two arcs of
 a domain $D$,  let
 \[    \exc_D(V_1,V_2) = \int_{V_1} \int_{V_2} \, H_D(z,w)
  \, |dz| \, |dw|.\]
This is $\pi$ times  the usual excursion measure between $V_1$
 and $V_2$; the factor of $\pi$ comes from our choice
 of  Poisson
 kernel.
Note that
\[   \exc_\Half((-\infty,0], [x,1])
   =\int_x^1  \int_{-\infty}^0 \frac{ dr\, ds}
                { (s-r)^{2}}  = \int_x^1 \frac{dr}{r} =
                   \log(1/x),\]
Hence we can write \eqref{mar18.4} as 
$ \phi\left(\exp\left\{-\exc_\Half((-\infty,0], [x,1])
 \right\} \right) .$
  More generally,
 if $x_1  < y_1 < x_2 < y_2$,
 \[   \tilde \Psi(\x,\y)
  = \phi\left(\exp\left\{-\exc_\Half([x_1,y_1], [x_2,y_2])
 \right\} \right)
    = \phi\left(\exp \left\{-\int_{x_1}^{y_1}
      \int_{x_2}^{y_2} \frac{dr\, ds}
       { (s-r)^2}\right\} \right),\]
and if $D$ is a simply connected
  subdomain of $\Half$ containing $x_1,y_2,x_2,y_2$
  on its boundary, then 
 \begin{equation}  \label{pat.1}
     \tilde \Psi_D(\x,\y)
     =  \phi\left(\exp\left\{-\exc_D( [x_1,y_1], [x_2,y_2])
 \right\} \right)
   =  \phi\left(\exp \left\{-\int_{x_1}^{y_1}
      \int_{x_2}^{y_2} H_D(r,s) \,{dr\, ds}
       \right\} \right).
       \end{equation}
  This expression is a little bulky but it allows for easy
  differentiation with respect to $x_1,x_2,y_1,y_2$.
       
       At this point we can state the main proposition.

\begin{proposition}  \label{mainprop}
$\Psi$ and $\tilde \Psi$ are $C^2$ functions.
\end{proposition}

It clearly suffices to prove this for $\tilde \Psi$.  The idea is
simple --- we will write the partition function as an expectation
and differentiate the expectation by interchanging the
expectation and the derivatives.  This interchange requires
justification and this is the main work of this paper.

We will use the following fact which is an analogue of
derivative estimates for   positive harmonic functions. 
 The proof is straightforward but we delay
it to Section \ref{lemmasec}.

\begin{lemma} \label{mar19.lemma1}
There  exists $c < \infty$ such that for every $x_1 < y_1 < x_2 < y_2$  the following holds.

\begin{itemize}

\item Suppose $D \subset  \Half $ is  
a simply connected domain whose boundary contains an open
real neighborhood of  $[x_1,y_1]$
and suppose that 
\[   \delta := \min\left\{\lvert x_1-y_1\rvert,\dist\left[\{x_1,y_1\}, \Half \setminus D
\right]\right\} >0.\]
Then if $z_1,z_2 \in \{x_1,y_1\},$
\[   |\p_{z_1}  H_D(x_1,y_1)|
   \leq c\, \delta^{-1} \, H_D(x_1,y_1).\]
\[    |\p_{z_1z_2}   H_D(x_1,y_1)|
   \leq c\, \delta^{-2} \, H_D(x_1,y_1).\]

\item
Suppose $D \subset  \Half $ is  
a simply connected domain whose boundary contains
open real neighborhoods of $[x_1,y_1]$
and $[x_2,y_2]$ and suppose that  
 \[
 \delta:=\min\left\{\{|w_1-w_2|;\,w_1\neq w_2\text{ and } w_1,w_2\in\{x_1,x_2,y_1,y_2\}\},\,\,\dist\left[\{x_1,y_1,x_2,y_2\}, \Half \setminus D\right]
 \right\}.
 \]
Then if $z_1 \in \{x_1,y_1\}, z_2 \in \{x_2,y_2\}$,
\[    |\p_{z_1z_2} \tilde \Psi_D(\x,\y)|
   \leq c\, \delta^{-2} \, \tilde \Psi_D(\x,\y).\]

\end{itemize}
Moreover, the constant can be chosen uniformly in neighborhoods of $x_1,y_1,x_2,y_2$.
\end{lemma} 

We will also need to show that expectations do not blow
up when paths get close to starting points.  We 
prove this lemma in Section \ref{japansec}.
Let
\[   \Delta_{j,k}(\bgamma) =
 \dist\left\{
    \{x_k,y_k\}, \gamma^j 
    \right\},\]
 \[ \Delta (\bgamma) = \min_{j \neq k}
      \Delta_{j,k}(\bgamma).\]

\begin{lemma}  If $\kappa < 4$, then
for every $n$ and every $(\x,\y)$, there exists
$c < \infty$ such that for all $\epsilon > 0$, and all $j \neq k$,
\[    \E\left[  Y;  \Delta \leq \epsilon
\right] \leq c \, \epsilon^{\frac{12}{\kappa}-1}. \]
In particular,
\[    \E\left[Y \, \Delta^{-2} \right] 
 \leq \sum_{m=-\infty}^\infty 2^{-2m} \, \E\left [Y;
   2^{m} \leq \Delta < 2^{m+1} \right]< \infty . \]
\end{lemma}

\begin{proof}  It suffices to show that for each $j,k$,
\[    \E\left[  Y;  \Delta_{j,k} \leq \epsilon\right] \leq c \, 
\epsilon^{\frac{12}{\kappa}-1}, \]
and by symmetry we may assume $j=1,k=2$.  If
we write $\bgamma = (\gamma^1,\gamma^2, \bgamma')$, then
the event  $\{\Delta_{1,2} \leq \epsilon\}$ is measurable
with respect to $(\gamma^1,\gamma^2)$ and 
\[   \E[Y \mid  \gamma^1, \gamma^2] \leq  Y(\gamma^1,\gamma^2).\]
Hence it suffices to prove the result when $n=2$.  This 
will be done in Section \ref{japansec}; in that section
we consider $\kappa < 8$.
\end{proof}

   For $n=1,2$,
it is clear that $\tilde \Psi$ is $C^\infty$ from 
  the exact expression, so we
will assume that $n \geq 3$.  By invariance under
permutation of indices, it suffices to consider
second order derivatives involving only $x_1,x_2,y_1,y_2$.
We will assume $x_j < y_j$ for  $j = 1,2$ 
and $x_1 < x_2$ (otherwise we just
relabel the vertices).  The configuration $x_1 < x_2 < y_1 < y_2$
is impossible for topological reasons.  If $x_1 < x_2 < y_2 < y_1$,
we can find a M\'obius transformation taking a point $y' \in (y_2,y_1)$
to $\infty$ and then the images would satisfy
$y_1' < x_1' < x_2'  < y_2'$ and this reduces to above.  So
we may assume that
\[    x_1 < y_1 < x_2 < y_2.\]

\noindent {\bf Case I:} Derivatives involving only $x_j,y_j$
for some $j$.

\medskip

We assume $j=1$. We will write
$\x = (x,\x'), y = (y,\y'), \bgamma = (\gamma^1,\bgamma')$, and 
let $D$ be the connected component
of $\Half \setminus \bgamma'$ containing $x,y$
on the boundary.  Then
\[      \E[ Y \mid \bgamma']
      =   Y(\bgamma') \, \left[\frac{H_{D}(x,y)}
     {
         H(x,y)}\right]^b 
   =  Y(\bgamma') \, Q_D(x,y)^b,\]
   where
  $Q_D(x,y)$  is the probability that a (Brownian)
   excursion in $\Half$
 from $x$ to $y$ stays in $D$.  
Hence
\[   \tilde \Psi(\x,\y) =
    \E  \left[ Y(\bgamma') \, Q_D(x,y)^b \right]
         .\]
Let $\delta = \delta(\bgamma') = \dist\{\{x,y\},\bgamma'\}.$
  Using  Lemma \ref{mar19.lemma1},  we see that
 \[    \left|\p_{x} [Q_D(x,y)^b] \right|
   \leq c \, \delta^{-1} \, Q_D(x_1,y_1)^b , \]  
   \[   \left|\p_{xy} [Q_D(x,y)^b]\right|
      +  \left|\p_{xx} [Q_D(x,y)^b]\right|
   \leq c \, \delta^{-2} \, Q_D(x_1,y_1)^b.\]
(Here  $c$ may depend on $x,y$ but not on $D$).
Hence
\[     \E\left[ Y(\bgamma') \,  \left |\p_x [Q_D(x,y)^b ]
 \right| \right] \leq c\,\E\left[ Y(\bgamma') \,
          \delta(\bgamma')^{-1}\, Q_D(x,y)^b\right] ,\]
         and if $z $ = $x$ or $y$,
 \[     \E\left[ Y(\bgamma') \,  \left |\p_{xz} [Q_D(x,y)^b ]
 \right| \right] \leq  c\,\E\left[ Y(\bgamma') \,
          \delta(\bgamma')^{-2}\, Q_D(x,y)^b\right] .\]
Since
   \[ 
\E\left[ Y(\bgamma') \,
          \delta(\bgamma')^{-2}\, Q_D(x,y)^b\right] =
        \E\left[\E\left(Y \, \delta^{-2} \mid \bgamma'\right)\right]
      = \E[Y\, \delta^{-2}]  \leq \E[Y \, \Delta^{-2} ] < \infty,  \]
  the interchange of expectation and derivative is valid,
  \[      \p_x \tilde \Psi (\x,y) =
     \E\left[Y(\bgamma') \, \p_x[Q_D(x,y)^b] \right], \;\;\;\;
   \p_{xz} 
\tilde \Psi (\x,y) =
     \E\left[Y(\bgamma') \, \p_{xz}[Q_D(x,y)^b] \right]. \]
\medskip

\noindent {\bf Case 2:}  The partial $\p_{z_1z_2}$
where $z_1 \in \{x_j,y_j\}, z_2 \in \{x_k,y_k\}$
with $j\neq k$.

\medskip

We assume $j=1,k=2$.
 We will write
$\x = (x_1,x_2,\x'), y = (y_1,y_2,\y'), \bgamma = (\gamma^1,\gamma^2,
\bgamma')$.   We will write $D' = D \setminus \bgamma'$
and let $D_1,D_2$ be the connected components of $D'$ containing
$\{x_1,y_1\}$ and $\{x_2,y_2\}$ on the boundary.  It is possible
that $D_1 = D_2$ or $D_1 \neq D_2$.
\begin{itemize}
\item  If $D_1 \neq D_2$, then
\[     \E[  Y \mid \bgamma']
      =   Y(\bgamma') \, Q_{D_1}(x_1,y_1)
      ^b \, Q_{D_2}(x_2,y_2)
      ^b 
       .\]
\item  If $D_1 = D_2=D$, then
\[   \E[  Y \mid \bgamma']
      =  Y(\bgamma') \, Q_{D_1}(x_1,y_1)
      ^b \, Q_{D_2}(x_2,y_2)
      ^b  \,\tilde \Psi_D((x_1,x_2),
       (y_1,y_2)), \]
 where $\tilde \Psi_D$ is defined as in \eqref{pat.1}.

\end{itemize} 

In either case we have written
\[  \E[  Y \mid \bgamma']
      =   Y(\bgamma') \, \Phi(\z;\bgamma'),\]
 where $\z = (x_1,y_1,x_2,y_2)$ and  we can use
 Lemma \ref{mar19.lemma1} to see that
 \[   \left| \p_{z_1z_2}\Phi(\z;\bgamma')
 \right| \leq c \, \Delta(\bgamma,\z)^{-2} \, \Phi(\z,
 \bgamma'), \;\;\;\; \Delta(\bgamma,\z)
  = \dist\{\gamma,\{x_1,y_1,x_2,y_2\}\}.\]
  As in the previous case, we can now interchange
  the derivatives and the expectation.

\section{Estimate}   \label{japansec}

In this section we will derive an estimate for $SLE_\kappa, \kappa < 8$.
While the estimate is valid for all $\kappa < 8$, the result is only
strong enough to prove our main result for $\kappa < 4.$
We follow the ideas in \cite{Japan} where careful analysis was made
of the boundary exponent for $SLE$.
Let $g_t$ denote the usual conformal transformation associated to the  $SLE_\kappa$ path $\gamma$ from $0$ to $\infty$
 parametrized so that
 \begin{equation}  \label{mar18.1}
     \p_t g_t(z) = \frac{a}{g_t(z) - U_t} , 
     \end{equation}
 where $a = 2/\kappa$ and  $U_t = - W_t$ is a standard Brownian motion.
 Throughout, we assume that $\kappa < 8$, so that
 $D= D_\infty = \Half \setminus \gamma$ is a nonempty
 set. If $0 < x < y < \infty,$ we let
 \[  \Phi = \Phi(x,y) = \frac{H_D(x,y)}
   {H_\Half(x,y)}, \]
 where $H$ denotes the boundary Poisson kernel.  If
 $x$ and $y$ are on the boundary of different components
 of $D$ (which can only happen for $4 < \kappa < 8$), then
 $H_D(x,y) = 0$.   As usual, we let
 \[   b = \frac{6-\kappa}{2\kappa} = \frac{3a-1}{2}.\]
As a slight abuse of notation, we will write $\Phi^b$ for
$\Phi^b \, 1\{\Phi > 0\}$ even if $b \leq 0$. 
 
\begin{proposition}  For every $\kappa  < 8 $ and $\delta > 0$,
there exists $ 0 < c < \infty$ such that for all
$\delta \leq x < y \leq 1/\delta$ and all $0 < \epsilon
<  (y-x)/10 $,
\[
   \E\left[\Phi^b;   \dist(\{x,y\},
   \gamma) < \epsilon \right]  \leq c \, \epsilon^{6a-1}.\]
 \end{proposition}

  It is already known that
  \[ 
   \Prob\left\{  \dist(\{x,y\},
   \gamma) < \epsilon  \right\}  \asymp  \epsilon^{4a-1},\]
 and hence we can view this as the estimate 
 \[   \E\left[\Phi^b  \mid  \dist(\{x,y\},
   \gamma) < \epsilon \right] \leq c \, \epsilon^{2a}.\]
  Using reversibility \cite{MS,Zhan}
   and scaling  of $SLE_\kappa$ we can see that
 to prove the proposition it suffices to show that for every $\delta >0$
 there exists $c = c_\delta$ such that if $\delta \leq x < 1$,
  \[  
   \E\left[\Phi^b;  \dist(1,
   \gamma) < \epsilon \right]  \leq c \, \epsilon^{6a-1}.\]
 This is the result we will prove. 
 
 \begin{proposition}  If $\kappa < 8$, there exists $c < \infty$
 such that if $\gamma$ is an $SLE_\kappa$ curve from $0$ to $\infty$,
 $0 < x  < 1$, $\Phi = \Phi(x,1)$, $0 < \epsilon \leq 1/2$,
 \[     \E\left[\Phi^b ; \dist(\gamma,1) < \epsilon\, (1-x) \right]
    \leq c \,x^a \, (1-x)^{4a-1}\, \epsilon^{6a - 1}.\]
  \end{proposition}

 We will relate the distance to the curve to a conformal radius.  In
 order to do this, we will need $1$ to be an interior point of the
 domain.   Let
 $D_t^*$ be the unbounded component of  
 \[   K_t = \Half \setminus \left[(-\infty,x] \cup \gamma_t \cup \{\bar z: z \in \gamma_t\}\right] ,\]
 and let 
 $T = T_1 = \inf\{t: 1 \not\in D^*_t\}$.  Then for $t < T$, the
 distance from $1$ to $\p D^*_t$ is the minimum of $1-x$ and
 $\dist(1,\gamma_t)$.   In particular, if $t < T$ and
  $\epsilon < 1-x$, then
 $\dist(\gamma_t,1) \leq \epsilon$ if and only if $\dist(1,\p D_t^*)
  < \epsilon$.  We define $\Upsilon_t$
  to  be $[4(1-x)]^{-1}$ times the 
   conformal radius of $1$ with respect to $D^*_t$ and
   $\Upsilon = \Upsilon_\infty$.
   Note that $\Upsilon_0 = 1$, and  if  $\dist(1,\p D^*_t)
  \leq  \epsilon(1-x)$, then $\Upsilon \leq \epsilon$.
   It suffices for us to   show that 
 \[     \E\left[\Phi^b ; \Upsilon < \epsilon \right]
    \leq c \, \epsilon^{6a - 1}.\]

We set up some notation.  We fix $0 < x < 1$ and
assume that $g_t$ satisfies \eqref{mar18.1}.  Let
\[   X_t = g_t(1) - U_t, \;\;\;\; Z_t = g_t (x) - U_t, \;\;\;\;
    Y_t = X_t - Z_t, \;\;\;\; {K_t} = \frac{Z_t}{X_t} , \]
  and note that the scaling rule for conformal radius implies
  that
  \[    \Upsilon_t = \frac{Y_t}{(1-x)\, g_t'(1)}.\]
  The Loewner equation implies that
  \[   dX_t = \frac{a}{X_t} \, dt + dB_t, \;\;\;\; dZ_t
   = \frac{a}{Z_t} \, dt + dB_t, \]
   \[     \p_t g_t'(1) = - \frac{a\, g_t'(1)}{X_t^2}, \;\;\; \p_t g_t'(x)
      = - \frac{a g_t'(x)}{Z_t^2}, \;\;\; \p_tY_t = -\frac{a \, Y_t}
        {X_t Z_t}
           . \]
            \[   \p_t \Upsilon_t = \Upsilon_t \, \,\left[ \frac{a}{X_t^2} - \frac{a}
{X_t Z_t} \right] = -a\Upsilon_t \, \frac{1}{X_t^2} \, \frac{1-K_t}{K_t} . \]
  Let $D_t$ be the unbounded component of $\Half \setminus \gamma_t$
  and let
  \[   \Phi_t = \frac{H_{D_t}(x,1)}{H_{D_0}(x,1)}
    =   x^2 \,\frac{g_t'(x)\, g_t'(1)}{Y_t^{ 2}},\]
  where we set $\Phi_t = 0$ if $x$ is not on the boundary of $D_t$, that is,
  if $x$ has been swallowed by the path (this is relevant only
  for $4 < \kappa < 8$).  Note that $\Phi =   \Phi_\infty$ and
  \[    \p_t \Phi_t^b = \Phi_t^b \, \left[ - \frac{ab }{X_t^2}
   - \frac{ab }{Z_t^2}+\frac{2ab }
        {X_t Z_t}\right]  = - ab \,  \frac{\Phi_t^b}{X_t^2}
         \, \left( \frac{1 - K_t}{K_t}\right)^2,\]
     \[ \Phi_t^b = \exp\left\{ -ab \int_0^t \frac{1}{X_s^2} \, 
          \left( \frac{1 - K_s}{K_s}\right)^2  \, ds \right\}.\]
     It\^o's formula implies
      that
   \[  d \frac{1}{X_t} = -\frac{1}{X_t^2} \, dX_t +
    \frac{1}{X_t^3} \, d\langle X \rangle_t =
      \frac1{X_t}  \, \left[\frac{1-a}{X_t^2} \, dt
          - \frac{1}{X_t} \, dW_t \right], \]
   and the product rule gives
   \[   d[1-K_t] = [1-K_t] \, \left[\frac{1-a}{X_t^2} \, dt -
    \frac a{X_t\,Z_t}  \, dt
          - \frac{1}{X_t} \, dW_t \right]
          = \frac{ 1-K_t }{X_t^2} \, 
           \left[(1-a) -
    \frac a{K_t} \right] \, dt - \frac{1-K_t}{X_t} \, dW_t.\]
   which can be written as
   \[   dK_t = \frac{1-K_t}{X_t^2} \, 
     \left[ 
    \frac a{K_t}  + a-1\right] \, dt + \frac{1-K_t}{X_t} \, dW_t.\]

As in \cite{Japan}, we consider the local martingale
\[  M_t^*=  (1-x)^{1-4a} \, X_t^{1-4a} \, g_t'(1)^{4a-1} = 
 (1-x)^{1-4a} \,  (1-K_t)^{4a-1} \, \Upsilon_t^{1-4a},\]
    which satisfies
\[                 dM_t^* = \frac{1-4a}{X_t} \, M_t^* \, dW_t,
\;\;\;\;M_0^* = 1\]
If we use Girsanov and tilt by the local martingale, we see
that
\[  dK_t =  \frac{1-K_t}{X_t^2} \, 
     \left[ 
    \frac a{K_t} -3a\right] \, dt + \frac{1-K_t}{X_t} \, d  W_t^*.\]
    where $  W_t^*$ is a standard
     Brownian motion in the new measure $\Prob^*$.
We reparametrize so that $\log \Upsilon_t$ decays
linearly.  More precisely, we let $\sigma(t)
 = \inf\{t: \Upsilon_t = e^{-at} \}$ and define
 $\hat X_{t} = X_{\sigma(t)}, \hat Y_t = Y_{\sigma(t)}$,
 etc.  Since  $\hat \Upsilon_t := \Upsilon_{\sigma(t)}
  = e^{-at}$, and 
  \[ - a \, \hat \Upsilon_t = \p_t \, \hat \Upsilon_t = -a \hat \Upsilon_t \,
  \frac 1 {\hat X_t^2}\, \frac{1-\hat K_t}{\hat K_t}
   \,  \dot \sigma(t), \]
we see that
   \[   \dot \sigma(t) = \frac{ \hat X_t^2\,  \hat K_t}{1-\hat K_t}, \]
  Therefore,
  \[  \hat \Phi_t^b :=  \Phi_{\sigma(t)}
  ^b = \exp \left\{-ab \int_0^t
      \frac{1-\hat K_s}{\hat K_s} \, ds \right\}
        = e^{abt} \, \exp \left\{-ab \int_0^t
      \frac{1}{\hat K_s} \, ds \right\}
         ,\]
    \begin{eqnarray*}
      d\hat K_t & = & \left[a - 3a \hat K_t\right]
     \, d t+ \sqrt{\hat K_t \, (1-\hat K_t)}
      \, dB_t^*\\
      & = & \hat K_t \, \left[\left(\frac{a}{\hat K_t}
        - 3a \right) \, dt + \sqrt{\frac{1-\hat K_t}
        {\hat K_t}} \, dB_t^*\right].
      \end{eqnarray*}
  for a standard Brownian motion $B_t^*$ (in the  measure $\Prob^*$).
     
Let $\lambda = 2a^2$, and
\[    N_t = e^{\lambda t} \,   \hat \Phi_t^b  \hat K_t^{a}
  = \exp\left\{\frac{a(7a-1)}{2} \, t\right\} \exp \left\{-\frac{a
  (3a-1)}2 \int_0^t
      \frac{1}{\hat K_s} \, ds \right\} \,   \hat K_t^{a} .\]        
It\^o's formula shows that $N_t$ is a local $\Prob^*$-martingale
satisfying
\[  d N_t = N_t \,  a\, \sqrt{\frac{1-\hat K_t}
        {\hat K_t}} \, dB_t^*, \;\;\;\;
           N_0 = x^a\]
One can
 show it is a martingale by using Girsanov to see that
\[  d \hat K_t = \left[2a - 4a \hat K_t\right]
     \, d t+ \sqrt{\hat K_t \, (1-\hat K_t)}
      \, d\tilde B_t,\]
  where $\tilde B_t$ is a Brownian motion in the
  new measure $\tilde \Prob$.  By comparison with a Bessel
  process, we see that the solution exists for all time.
  Equivalently, we can say that
  \[    \hat  M_t := \hat M_t^* \, N_t , \]
  is a $\Prob$-martingale with $\hat M_0
   = x^a$.  (Although $M_t^*$ is only a local
  martingale, the time-changed version $\hat M_t^*:=
  M_{\sigma(t)}^*$ is a martingale.)

  Using \eqref{mar18.4}  we see that $\E\left[\Phi_\infty^b \mid
  \gamma_{\sigma(t)}\right]  \leq c \, \hat K_t^a \,  \hat \Phi_t^b .$
  If $\epsilon = e^{-at}$, then 
  
\begin{eqnarray*}
\E\left[   \Phi^b ; \sigma(t) < \infty \right]
 & = & c\,\E\left[\E( \Phi^b  \, 1\{ \sigma(t) < \infty\}  \mid \gamma_{\sigma(t)})\right]\\
& \leq & c\, \E\left[\hat K_t^a \, \hat \Phi^b_t ; \sigma(t) < \infty \right]\\
& = & c\,e^{-\lambda t} \, e^{(1-4a)at} \, (1-x)^{4a-1}
  \hat M_0^{-1} \,
\E\left[ \, \hat M_t \; (1-\hat K_t)^{1-4a} 
 ; \sigma(t) < \infty \right]\\
& = &c\, e^{a(1-6a)t} \,x^a\, (1-x)^{4a-1}\, \tilde \E\left[  (1-\hat K_t)^{1-4a}
\right]\\
&  = &c\, \epsilon ^{6a-1} \, x^a\, (1-x)^{4a-1} \, 
 \tilde \E\left[  (1-\hat K_t)^{1-4a}
\right].
\end{eqnarray*}

So the result follows once we show that
\[    \tilde \E\left[  (1-\hat K_t)^{1-4a}
\right]< \infty\]
is uniformly bounded for $t \geq t_0$.  The argument
for this proceeds as in \cite{Japan}.
If we do the
change of variables $\hat K_t = [1 - \cos \Theta_t]/2$, then
It\^o's formula shows that
\[   d \Theta_t = \left(4a - \frac 12\right) \, \cot \Theta_t\, dt  + dB_t.\]
This is a radial Bessel process that  never reaches the boundary.
It is known that the invariant distribution is proportional to $\sin^{8a-1}
 \theta$ and that  it approaches the invariant distribution
exponentially fast.  One then computes that the invariant distribution
for $\hat K_t$ is proportional to $x^{4a-1} \, (1-x)^{4a-1}$.
In particular, $ (1-\hat K_t)^{1-4a}$ is integrable with respect
to the invariant distribution.

 \section{Proof of Lemma \ref{mar19.lemma1}}  \label{lemmasec}
 
 We prove the first part of Lemma \ref{mar19.lemma1} for $x_1=0,\,y_1=1$. Other cases follow from this and a M\"obius transformation sending $x_1,y_1$ to $0,1$.
 \begin{lemma}   \label{Hlemma}
 There exists $c< \infty$ such that
 if  $D$ is a simply connected
 subdomain of $\Half$ containing $0 ,1$ on its boundary, then
 \[   |\p_x H_D(0,1)| + |\p_y H_D(0,1)|
  \leq c \, \delta^{-1} \, H_D(0,1),\]
  \[ |\p_{xx} H_D(0,1)| + |\p_{xy} H_D(0,1)|
    + |\p_{yy} H_D(0,1)| \leq c\, \delta^{-2} \, H_D(0,1),\]
 where $\delta = \dist(\{0,1\}, \p D \cap \Half)$. 
 
 \end{lemma}
\begin{proof}
Let $g: D \rightarrow \Half$ be a conformal transformation
 with $g(0) = 0, g(1) = 1, g'(0) = 1$.  Then if $|x| < \delta,
 |y-1|< \delta$, 
 \begin{equation} \label{mar18.3}
   H_D(x,y) = \frac{g'(x) \, g'(y)}{ [g(y) - g(x)]^{2}}.
   \end{equation}
 In particular $g'(0) \, g'(1) = H_D(0,1) \leq H_\Half(0,1)
  = 1$ and hence $g'(1) \leq 1$.
   Using Schwartz reflection we can extend $g$ to be a
   conformal transformations of disks of radius $\delta$
   about $0$ and $1$.  
    By the distortion estimates (the fact that $|a_2| \leq 2,
    |a_3| \leq 3$ for schlicht functions) we have
  \[   |g''(0)| \leq 4 \, \delta^{-1} \, g'(0)
   \leq 4 \, \delta^{-1}, \;\;\;\;
     |g'''(0) | \leq 18 \, \delta^{-2} \, g'(0)
       \leq 18 \, \delta^{-2} ,\]
  and similarly $|g''(1)| \leq 4 \, \delta^{-1}\, g'(1)\,$ and $
  |g'''(1)| \leq 18 \, \delta^{-1}\, g'(1)$.  By direct differentiation
  of the right-hand side of \eqref{mar18.3}
  we get the result.
    \end{proof}

 \begin{lemma}  There exists $c < \infty$ such that if
 \[ x_1 < y_1 \leq 0 < 1 \leq x_2 < y_2 ,\]
 $\tilde \Psi_D(\x,\y)$ is as in \eqref{pat.1},
 and $z_1 \in \{x_1,y_1\}, z_2 \in \{x_2,y_2\}$,
 then
 \[    |\p_{z_1} \tilde \Psi_D(\x,\y)|
    + |\p_{z_2} \tilde \Psi_D(\x,\y) |
        \leq c \, \delta^{-1} \, \Psi_D(\x,\y),\]
  \[   |\p_{z_1z_2} \tilde \Psi_D(\x,\y)|
     \leq c \, \delta^{-2} \, \tilde \Psi
       _D(\x,\y),\]
 where
 \[
 \delta:=\min\left\{\{|w_1-w_2|;\,w_1\neq w_2\text{ and } w_1,w_2\in\{x_1,x_2,y_1,y_2\}\},\,\,\dist\left[\{x_1,y_1,x_2,y_2\}, \Half \setminus D\right]
 \right\}.
 \]
 \end{lemma}

 \begin{proof}
 Let   
\[  \tilde \Psi_D(\x,\y)
     =   \phi\left(u_D(\x,\y) \right). \]
where 
\[   u_D (\x,\y) = e^{-\exc_D(\x,\y)}
, \;\;\;\; \exc_D(\x,\y) = 
\int_{x_1}^{y_1}
      \int_{x_2}^{y_2} H_D(r,s) \,{dr\, ds}.\]

Using the Harnack inequality we can see that
for $j=1,2$, 
\[       H_D(x,s)  \asymp H_D(x_j,s),\;\;\;\;
    H_D(r,y) \asymp H_D(r, y_j) \]
if $|x-x_j| \leq \delta/2, |y-y_j| \leq \delta/2$.
From this we see that  
 \[    \int_{x_2}^{y_2} H_D(z_1,s) \, ds
  +  \int_{x_1}^{y_1} H_D(r,z_2) \, ds
     \leq c \, \delta^{-1} \, \exc_D(\x,\y),\]
  \[   H_D(z_1,z_2) \leq 
      c \, \delta^{-2} \, \exc_D(\x,\y).\]
 Let $z_1$ be $x_1$ or $y_1$ and let $z_2$ be
 $x_2$ or $y_2$. Then,
 \[  \p_{z_1} \, \tilde \Psi_D(\x,\y)
    =   {\phi'( u_D (\x,\y))}
        \,\p_{z_1} u_D(\x,\y)  \]
  \[ \p_{z_1z_2} \tilde \Psi_D(\x,\y)
    = \phi''( u_D (\x,\y))
        \,[\p_{z_1} u_D(\x,\y) ]\,[
        \p_{z_2} u_D(\x,\y)]
      +\phi'(u_D (\x,\y))\,\p_{z_1z_2}u_D(\x,\y).
        \]
 \[  \p_{z_1}  u_D(\x,\y) =
  \left[\pm \int_{x_2}^{y_2} H_D(z_1,s) \, ds\right]\,  u_D(\x,\y) .\]
  \[ \p_{z_2}  u_D(\x,\y) =  
   \left[\pm \int_{x_1}^{y_1} H_D(r,z_2) \, ds\right]\,  u_D(\x,\y).\]
 \[  \p_{z_2z_1}  u_D(\x,\y) =    \left[\pm \int_{x_1}^{y_1}  \int_{x_2}^{y_2}H_D(r,z_2) \, dr
    \, H_D(z_1,s) \, ds
    \pm H_D(z_1,z_2)\right] \,  u_D(\x,\y) \]
  This gives
 \[  | \p_{z_1}  u_D(\x,\y)| + |\p_{z_2}  u_D(\x,\y)|
   \leq c\, \delta^{-1}  \,\exc _D(\x,\y) \, u_D(\x,\y),\]
  \[  \left|\p_{z_2z_1}  u_D(\x,\y)\right|
    \leq c\, \delta^{-2}  \,\exc _D(\x,\y) \, u_D(\x,\y).\]
  \[ \frac{| \p_{z_2z_1}  u_D(\x,\y)|}
  { u_D(\x,\y)} \leq c \, \delta^{-2} \, \exc_D(\x,\y).\]
  The result will follow if we show that 
   \[     \frac{(1-x) \, |\phi''(1-x)|}{\phi(x)},\;\;\;\;
       \frac{\phi'(x)}{\phi(x)},\]
    are uniformly bounded for $x >x_0$.
    
Recall that $u(x) = c\, x^a \, F(x)$ where 
  $  F(x) =\,_2F_1(2a,1-2a,4a;x)$.  We recall that $F$
  is analytic in the unit disk with power series
  expansion
  \[          F(x) = 1 + \sum_{n=1}^\infty b_n \, x^n ,\]
  where the coefficients $b_j$ satisfy
  \[                 b_n =   C \, n^{4a-2} \, [1 + O(n^{-1})].\]
    We therefore get asymptotic expansions for the coefficients
    of the derivatives of $F$.
The important thing for us is that if $\kappa < 8$, then
$4a-1 > 1$ and we have as $x\downarrow 1$
\[  F(1-x) =O(1),\;\;\;\; F'(x) = o(x^{-1}), \;\;\;
  F''(x)  o(x^{-2}).\]
  In other words, the quantities
  \[    F(x) , \;\;\;\frac{(1-x) \, F'(x)}{F(x)}, \;\;\;\;
     \frac{(1-x)^2 \, F''(x)}{F(x)}, \]
 are uniformly bounded for $0 \leq x < 1$.  If
 $g(x) = x^a \, F(x)$, then
 \[  g'(x) = g(x) \, \left[\frac{a}{x} + \frac{F'(x)}{F(x)}
  \right],\]
  \[   g''(x) = g(x) \, \left[\left(  \frac{a}{x} + \frac{F'(x)}{F(x)}
  \right)^2
       - \frac{a}{x^2} + \frac{F''(x)}{F(x)} - \frac{F'(x)^2}{
        F'(x)^2}\right].\]
 Therefore, for every $x_0 >0$, the quantities
 \[ \phi(x) , \;\;\;\frac{(1-x) \, \phi'(x)}{\phi(x)}, \;\;\;\;
     \frac{(1-x)^2 \, \phi''(x)}{\phi(x)}, \]
 are uniformly bounded for $x_0 < x < 1$.
  \end{proof}


\begin{thebibliography}{00}
  
  \bibitem{Dub}  J. Dub\'edat (2006).   Euler integrals for
  commuting SLEs, J. Stat. Phys {\bf 123}, 1183--1218.
  
  \bibitem{KL} M. Kozdron, G. Lawler (2007). The configurational measure pn mutuallly avoiding SLE
  paths, in {\em Universality and Renormalization: From Stochastic
  Evolution to Renormalization of Quantum Fields}, I. Binder, D Kreimer, ed., Amer. Math. Soc., 199--224.
  
  \bibitem{Parkcity}  G. Lawler (2009). 
  Schramm-Loewner Evolution, in {\em Statistical Mechanics}, S. Sheffield, T Spenceer, ed., IAS/Park City Mathematics Series {\bf 16}, 233--395.
  
  \bibitem{Japan}  G. Lawler (2015).  Minkowski content of the intersection
  of a Schramm-Loewner evolution (SLE) curve with the real  line, 
  J. Math Soc. Japan {\bf 67}, 1631--1669.
  
  \bibitem{MS} J. Miller, S. Sheffield (2016).
   Imaginary geometry III: reversibilty
  of $SLE_\kappa$ for $\kappa \in (4,8)$, Annals of Math. {\bf 184},  455--486.
  
  
  \bibitem{Oded}
   O. Schramm (2000).  Scaling limits for loop-erased random walks
  and uniform spanning trees, Israel J. Math {\bf 118}, 221--288.
  
  \bibitem{Zhan}  D. Zhan (2008). Reversibility of chordal $SLE$, Annals of
  Prob. {\bf 4}, 1472--1494.
  
   \bibitem{Eve}  E. Peltola, H. Wu (2017). Global Multiple $\sle$ for $\kappa\leq4$ and Connection Probabilities for Level Lines of GFF, \emph{Preprint in} arXiv:1703.00898
 
    \bibitem{Annulus}  G. Lawler (2011).  Defining $SLE$ in multiply connected domains with the Brownian loop measure, arXiv:1108.4364
 
     \bibitem{Stat}  G. Lawler (2009).  Partition Functions, Loop Measure, and Versions of $SLE$, Journal of Statistical Physics {\bf 134},   813--837
 

 \end{thebibliography}
\end{document}